\documentclass[12pt]{amsart}
\usepackage{latexsym, amsmath, amssymb}
\usepackage{hyperref}

\theoremstyle{plain}
  \newtheorem{Thm}{Theorem}[section] 
  \newtheorem{Lma}[Thm]{Lemma} 
  \newtheorem{Cor}[Thm]{Corollary} 
  \newtheorem{Prop}[Thm]{Proposition}

\theoremstyle{definition}
  \newtheorem{Def}[Thm]{Definition}

\theoremstyle{remark}
  \newtheorem{Rem}[Thm]{Remark}

  \newtheorem{Conj}[Thm]{Conjecture}

\newcommand{\brq}{^{[q]}}

\newcommand{\mlabel}[1]%
  {\mbox{}\marginpar{\raggedleft\hspace{0pt}{\rm\ttfamily#1}}\label{#1}}

\newcommand{\length}{\operatorname{\lambda}}

\newcommand{\Hom}[3]{\operatorname{Hom}_{#1}(#2,#3){}}
\newcommand{\R}[1]{{R^{({#1})}}}
\newcommand{\num}[2]{{\#_{#1}(#2)}}

\newcommand{\ta}{s}

\newcommand{\fm}{{\mathfrak m}}
\newcommand{\fn}{{\mathfrak n}}

\newcommand{\mx}{{\boldsymbol x}}
\newcommand{\my}{{\boldsymbol y}}

\newcommand{\ringS}{\text{$(S,\fn,k)$}}

 \newcommand{\height}{{\rm ht}}

\newcommand{\Dim}{{\rm dim}}

\newcommand{\Spec}{{\rm Spec}}

\newcommand{\im}{{\rm Im}}

\renewcommand{\ge}{\geqslant} \renewcommand{\le}{\leqslant} 
\renewcommand{\geq}{\geqslant} \renewcommand{\leq}{\leqslant} 

\newcounter{hours}\newcounter{minutes}

\newcommand{\excise}[1]{}
\begin{document}

\title[The lower semicontinuity of the Frobenius splitting numbers]
{\bf The lower semicontinuity\\ of the Frobenius splitting numbers}
\author[Florian~Enescu]{Florian Enescu}
\author[Yongwei~Yao]{Yongwei Yao}
\address{Department of Mathematics and Statistics, Georgia State
  University, Atlanta, GA 30303 USA} 
\email{fenescu@gsu.edu}
\email{yyao@gsu.edu}
\thanks{The first author was partially supported by the NSA grant
  H98230-07-10034.  
The second author was partially supported by the NSF grant DMS-0700554.}
\date{}


\begin{abstract}
We show that, under mild conditions, the (normalized) Frobenius
splitting numbers of a local ring of prime characteristic are lower
semicontinuous. 
\end{abstract}
\maketitle 

\section{Introduction and terminology} 

Throughout this paper, all rings are assumed to be commutative
Noetherian of positive  characteristic $p$ with $p$ prime (unless
stated otherwise explicitly).  
Let $q=p^e$ denote a power of the characteristic of the ring with $e \ge 0$. 
By a local ring $(R, \fm, k)$ we mean a Noetherian ring $R$
with only one maximal ideal $\fm$ and the residue field $k$.  

In recent years, a number of authors have studied a sequence of
numbers associated to a local ring $(R, \fm, k)$ of prime
characteristic $p>0$, called here the Frobenius splitting numbers,
that arise naturally in connection to the Frobenius homomorphism $F: R
\to R$, $F(r) =r^p,$ for all $r \in R$.   

Let us assume that $R$ is reduced and denote by $R^{1/q}$ the ring of
$q^{th}$ roots of elements in $R$ where $q=p^e$ with $e\geq
0$. Further assume that $R$ is F-finite, which by definition means
that $R^{1/q}$ is module finite  over $R$ for all $q$. Write 
$$R^{1/q} \cong R^{\oplus a_e} \oplus M_e,$$ 
which is a direct  sum decomposition
of $R^{1/q}$ over $R$ such that $M_e$ has no free direct summands. The
number $a_e$ is called the {\it $e^{th}$ Frobenius splitting number of
  $R$} and much work has been dedicated to investigating the size of
these numbers as in~\cite{HL, Ab,AE, AE2, Si,Y2, Y}. They are
intimately connected to the notions of F-purity and strong
F-regularity and in fact they can be defined more generally for local
rings of prime characteristic, F-finite or not. 

Our main result of the paper states that under mild conditions these
numbers are lower semicontinuous, and therefore they exhibit a natural
geometric behavior. In fact, we conjecture that the lower
semicontinuity of these numbers holds for all excellent locally
equidimensional rings.

We will now proceed to define the (normalized) Frobenius splitting
numbers of a (not necessarily F-finite or reduced) Noetherian local
ring $(R, \fm, k)$. 

For any $e \geq 0$, we let $R^{(e)}$ be the $R$-algebra defined as
follows: as a ring $\R{e}$ equals $R$ while the $R$-algebra structure
is defined by $r \cdot s = r^{q} s$, for all $r \in R,\, s \in
R^{(e)}$. Note that when $R$ is reduced 
we have that $\R{e}$ is isomorphic to $R^{1/q}$ as $R$-algebras. Also,
$\R{e}$ as an $\R{e}$-algebra is simply $R$ as an $R$-algebra. 
For example, given an ideal $I$ of $R$, we have $R/I \otimes_R \R{e}$
is (naturally isomorphic to) $R/I^{[q]}$, in which $I^{[q]}$ is the
ideal of $R$ generated by $\{x^q : x \in I\}$. 

Let $E=E_R(k)$ denote the injective hull of the residue field $k =
R/\fm$. We have a natural short exact sequence of $R$-modules: 
$$ 0 \to k \stackrel{\psi}\to E \stackrel{\phi}\to E/k\to 0,$$ 
which induces an exact sequence
$$\R{e} \otimes _R k \stackrel{1\otimes\psi}\longrightarrow 
\R{e} \otimes_R E \stackrel{1\otimes\phi}\longrightarrow 
\R{e} \otimes E/k \longrightarrow 0.$$

One can see that $K_e:=\im (1 \otimes_R \psi)$ is finitely generated
over $\R{e}$ and killed by $\fm\brq$. Therefore it has finite length
as an $\R{e}$-module. 

Let $u$ be a socle generator of $E$. The reader can note that $K_e$ is
in fact the $\R{e}$-submodule of $\R{e} \otimes _R E$ generated by $1
\otimes u$.

\begin{Def}
\label{def}
Let $(R, \fm, k)$ be a local ring of positive prime characteristic $p >0$
and let $e \geq 0$. Also let $E= E(k)$, $\psi$,
$\phi$ and $K_e$ be as above.

\begin{enumerate}
\item the {\it $e^{th}$ normalized Frobenius splitting number of $R$},
denoted by $\ta_e(R)$, is defined as 
$$\ta_e(R) := \frac{\lambda_{\R{e}}(K_e)}{q^{\dim(R)}} 
= \frac{\lambda_{\R{e}} (\im (1 \otimes_R \psi))}{q^{\dim(R)}}.$$ 

\item in the particular case when $R$ is F-finite, let $\alpha(R) = \log_p
[k:k^p]$ so  that $[k:k^q] = q^{\alpha(R)}$ for all $q = p^e$. Then
the {\it $e^{th}$ Frobenius splitting number of $R$}, denoted by
$a_e(R)$, is defined as 
$$a_e(R) = \ta_e(R) \cdot q^{\dim(R) +\alpha(R)}.$$

\item also, for a prime ideal $P$ in $R$, we will use $\ta_e(P)$ and
$a_e(P)$ to denote $\ta_e(R_P)$ and $a_e(R_P)$ respectively. 
\end{enumerate}
\end{Def}

The reader should note  that Yao has shown in \cite[Lemma~2.1]{Y} that
when $R$ is F-finite the numbers $a_e(R)$ are exactly the numbers
$a_e$ (mentioned earlier in the introduction) such that 
$$\R{e}= R^{\oplus a_e} \oplus M_e,$$ is a direct  sum decomposition
of $\R{e}$ over $R$ where $M_e$ has no free direct  summands.

Yao has showed the following characterization of regular local rings
which emphasizes the importance of these numbers in full generality,
see \cite[Lemma~2.5]{Y}. The corresponding result for F-finite rings
was known from work of Huneke and Leuschke~\cite{HL}.

\begin{Thm}
Let $(R, \fm, k)$ be a local ring of positive characteristic $p$, where $p$
is prime. Then $R$ is regular if and only if 
$$\ta_e (R) =1 \quad \text{for some (equivalently,
for all) $e \geq 1$.}
$$ 
\end{Thm}
 
One important point to make is that if $R$ is F-finite reduced
then the Frobenius splitting numbers can be defined directly as we did
at the beginning of our introduction, while in the general case it is
necessary to first consider the normalized Frobenius splitting
numbers.  

We remark that the limit of the sequence of normalized Frobenius
splitting numbers, when it exists, is a remarkable invariant of the ring,
namely the \emph{F-signature}. That is, the F-signature of $R$, denoted
$s(R)$, equals 
$$s(R): = \lim_{e \to \infty} \ta_e(R) \quad \text{if it exists.}$$  
In fact, this is why $s_e(R)$ is called the $e^{th}$ normalized Frobenius
splitting number.

In general, we have an \emph{upper F-signature} (respectively, 
\emph{lower F-signature})
of $R$ defined by $s^+(R) = \limsup_{e \to \infty} \ta_e(R)$
(respectively, $s^-(R)=\liminf_{e \to \infty}\ta_e(R)$). 
An important result states that for an excellent local ring $R$,
$s^+(R) >0$ if and only if $s^-(R) >0$ if and only if $R$ is strongly
F-regular (see \cite[Theorem~0.2]{AL} and \cite[Theorem 1.3 (2)]{Y}).
For more information on the Frobenius splitting numbers, the
F-signature and related concepts, we refer the reader to~\cite{HL, Ab,
AE, AE2, AL, Si, Y2, Y}.

\section{The conjecture}

We are now in position to state the aim of our paper.  We remind the
reader that a function $f: X \to \mathbb{R}$, with $X$ a topological
space, is \emph{lower semicontinuous} if the set 
$X_{\leq r} := \{ x \in X : f(x) \leq r \}$ 
is closed or, equivalently, $X_{> r} := \{ x \in X : f(x) > r \}$ 
is open for all $r \in \mathbb{R}$.

A ring $(R, \fm, k)$ is \emph{equidimensional} if $\dim(R/P) =\dim(R)$
for all minimal primes $P$ of $R$. A ring $R$ is \emph{locally}
equidimensional if 
$R_P$ is equidimensional for all $P \in \Spec(R)$. If $P \subseteq Q$
are prime ideals in an equidimensional and catenary ring $R$, then
$\height(Q) = \height(P) + \height(Q/P)$ 
(see \cite[Lemma~2 on page~250]{M}).

It is also helpful to remind the reader of the following notations: for
any ideal $P$ in $R$, we denote $V(I): = \{P  \in \Spec(R): I \subseteq
P\}$ which is a closed subset in $\Spec(R)$. For $x \in R$, we denote 
$D(x): = \{ Q \in \Spec(R) : x \notin Q \}$ which is an open subset of
$\Spec(R)$.

\begin{Conj}
\label{conj}
Let $R$ be a Noetherian ring of prime characteristic $p$ 
and fix $e\geq 0$.  
Let $s_e : \Spec(R) \to \mathbb{Q}$ be defined by
$$s_e(P) : = \ta_e(R_P), \;\forall P \in \Spec(R).$$
If $R$ is excellent and locally equidimensional, then $s_e$ is lower
semicontinuous. 
\end{Conj}

In this paper we will show that Conjecture~\ref{conj} holds true in
many significant cases.  In light of the fact that the F-signature and
the Hilbert-Kunz multiplicity of a ring exhibit at times parallel
behavior, it is perhaps interesting to note here 
that Shepherd-Barron proved in~\cite{S} that the Hilbert-Kunz
functions are upper semicontinuous for an excellent and locally
equidimensional ring of prime characteristic. 

First we note the following rather general fact.

\begin{Prop}
\label{modulecase}
Let $R$ be a Noetherian ring (not necessarily of prime characteristic
$p$) and $M$ a finitely generated $R$-module.
For a prime ideal $P$ of $R$, let $\num{P}{M}$ equal the maximal 
number of free copies of $R_P$ as direct summands in $M_P$. 
Consider the function 
$$\Spec(R) \to \mathbb{Q} \quad \text{defined
  by} \quad P \mapsto \num{P}{M}.$$
Then this function is lower semicontinuous.
\end{Prop}

\begin{proof}
Let $r \in \mathbb{R}$ and $P \in \Spec(R)$ such that $\num{P}{M} >
r$. Let $n =\num{P}{M}$. 
Then there exists a $R$-linear surjection 
$$M_P \to R_P^{\oplus n} \to 0.$$

Note that $(\Hom{R}{M}{R^n})_P \cong \Hom{R_P}{M_P}{R_P^{\oplus n}}$ and so
one can lift the above surjection to $\phi \in
\Hom{R}{M}{R^{\oplus n}}$, which gives the following 
exact sequence
$$M \stackrel{\phi}\longrightarrow R^{\oplus n} \longrightarrow C
\longrightarrow 0,
$$
in which $C$ is the cokernel of $\phi$. This forces $C_P=0$ as
$\phi_P$ is surjective. 

But $C_P=0$ implies that there exists $x \not\in P$ such that $C_x
=0$, which gives the following exact sequence
$$ M_x \stackrel{\phi_x}\longrightarrow R_x^{\oplus n} \longrightarrow 0.$$
Thus, for all $Q \in D(x) = \{ Q \in \Spec(R) : x \notin Q\}$, there is
an exact sequence
$$M_Q \longrightarrow R_Q^{\oplus n} \longrightarrow 0,$$
which implies that $\num{Q}{M} \geq n > r$ for all $Q$ in the open set
$D(x)$. 
\end{proof}

\begin{Lma}
\label{kunz}
Let $R$ be a ring of prime characteristic $p$, F-finite and locally
equidimensional. 

Then, on a connected component of $\Spec(R)$, the number ${\dim(R_P)
+\alpha(R_P)}$ is constant, in which $\alpha(R_P)$ is as described in
Definition~\ref{def}~(2). 

\end{Lma}
\begin{proof}
This was essentially proved by Kunz in \cite[Corollary~2.7]{K}. 
\end{proof}

\begin{Rem}
The reader should be aware that Kunz states his result under the
hypothesis that $R$ is equidimensional. In the generality stated
in~\cite{K} the result is not correct as  Shepherd-Barron showed
in~\cite{S}. 

The error in the proof of Kunz is in the last line of his proof where
he assumes without proof that $\height(Q)= \height(P) + \height(Q/P)$
for prime ideals $P \subseteq Q$ in $R$. For our Lemma stated above,
this follows from the condition that $R$ is locally equidimensional:
if we localize at $Q$ we get an equidimensional ring $R_Q$ and in an
equidimensional excellent local ring the relation holds as remarked at
the beginning of this section (cf. \cite[Lemma~2 on page~250]{M}).
\end{Rem}

As an immediate consequence we obtain:

\begin{Cor}
\label{F-finite}
Let $R$ be a F-finite locally equidimensional ring of positive
characteristic $p$, $p$ prime.  
Define $a_e:\Spec(R) \to \mathbb{Q}$ by $a_e(P) = a_e(R_P)$ (the
F-finite property localizes, so the definition is possible). 

Then the Frobenius splitting numbers and the normalized Frobenius
splitting numbers are lower semicontinuous, i.e., both $s_e$ and $a_e$
are lower semicontinuous functions for all $e$. Moreover, these
functions are proportional on each connected component of $\Spec(R)$
(with a possibly different factor of proportionality on each
component). 
\end{Cor}
\begin{proof}
As remarked earlier, we have the relationship 
$$a_e(R_P) = \ta_e(R_P) \cdot q^{\dim({R_P}) +\alpha(R_P)}, \; \forall P \in
\Spec(R) .$$

Since lower semicontinuity can be checked on each connected component
of $\Spec(R)$, we may assume $\Spec(R)$ is connected without loss of
generality.   

But then $a_e(P)= a_e(R_P)$ and $s_e(P)= \ta_e(R_P) $ are
proportional, as $\dim(R_P) +\alpha(R_P)$ is constant by
Lemma~\ref{kunz}. 
Hence it suffices to show $a_e$ is lower semicontinuous.

Also note that $a_e(R_P)$ is simply $\num{P}{\R{e}}$ as in
Proposition~\ref{modulecase}. Since $\R{e}$ is finitely generated over
$R$ because $R$ is F-finite, we can apply Proposition~\ref{modulecase}
and conclude that $a_e$ is lower semicontinuous.  This implies that
$s_e$ is lower semicontinuous as well by the fact that they are
related by proportionality. 
\end{proof}

This Corollary answers our Conjecture~\ref{conj} in the F-finite case.
Much more work will be needed if we are not under the presence of the
F-finite condition as our next sections will show. 

We will need the following criteria attributed to Nagata; see \cite{M2}.

\begin{Prop}
\label{nagata}
Let $R$ be a Noetherian ring and $U$ be a subset of $\Spec(R)$.
Then $U$ is open if and only if both of the following statements hold
\begin{enumerate}
\item[(i)] if $P \in U$, $Q \in \Spec(R)$ and $Q \subseteq P$, then $Q \in U$.
\item[(ii)] for all $P \in U$, $U \cap V(P)$ contains a nonempty open subset
of  $V(P)$; or equivalently, for all $P \in U$, there exists $x \in R
\setminus P$ such that $D(x) \cap V(P) \subseteq U$. 
\end{enumerate}
\end{Prop}

\section{Homomorphic images of regular rings}

Throughout this section we let $R=S/I$ where $S$ is a regular ring of
prime characteristic $p$ and $I$ an ideal of $S$. As always $q =p^e$. 

\begin{Prop}
\label{rewrite}
Let $R=S/I$ be a homomorphic image of regular local ring $\ringS$.
Then 
$$\ta _e(R) \cdot q^{\dim(R)} = \length_S \left(\frac{S}{\fn \brq
    :(I \brq :I)}\right) = \length _S\left( \frac{(I\brq :I) + \fn
    \brq}{\fn \brq}\right),$$ 
for any nonnegative integer $e$. 
\end{Prop}

\begin{proof}
We can make a faithfully flat extension of $S$ (and hence $R$) by first
completing and then enlarging the residue field of $S$ (and hence $R$) to
its algebraic closure. This flat local extension has its closed fiber
equal to a field. Note that the rings are F-finite after extension.

In \cite[Remark~2.3 (3)]{Y}, it is shown that the normalized Frobenius
splitting number is unchanged under such extensions. Also, $\length_S
\big(\frac{S}{\fn \brq :(I \brq :I)}\big)$ is not affected under
such an extension.   

So, it is enough to check the equality 
$\ta _e(R) \cdot q^{\dim(R)} = \length_S \big(\frac{S}{\fn \brq
    :(I \brq :I)}\big)$
under the additional hypothesis
that $R$ is F-finite and this is observed in~\cite{AE} (see the
remarks immediately after \cite[Theorem~4.2]{AE}). 
(The reader might be also interested in \cite{F, G}.)

Finally, the equality 
$\length _S \big(\frac{S}{\fn \brq :(I \brq :I)}\big) = 
\length _S\big( \frac{(I\brq :I) + \fn \brq}{\fn \brq}\big)$ holds
true by Matlis Duality as shown in~\cite[page 11]{AE}.  
\end{proof}

We will also need the following lemma. A proof is included for completeness.

\begin{Lma}
\label{filtration}
Let $R$ be a Noetherian ring, $P$ a prime ideal in $R$ and $L
\subseteq M$ be $R$-modules such that $M/L$ is finitely generated over
$R$. Assume that $\length_{R_P}(M_P/L_P)=n < \infty$.  

Then there exists $x \in R \setminus P$ and a filtration of $R_x$-modules
$$
L_x =M_0 \subseteq M_1 \subseteq \cdots \subseteq M_n =M_x
$$ 
such that $M_i/M_{i-1} \cong (R/P)_x$ for all $i =1, \ldots, n$.
\end{Lma}

\begin{proof}
Let $L= N_0 \subseteq N_1\subseteq \cdots \subseteq N_r =M$ be a prime
filtration of $L \subseteq M$ such that $N_i/N_{i-1} \cong R/Q_i$,
where $Q_i \in \Spec(R)$ for $i=1,\dotsc,r$.

For each of the indices $i$ such that $Q_i \not\subseteq P$, choose
$x_i \in Q_i$ but not in $P$. Let $x$ be the product of all
these elements $x_i$. 
Tensoring the original filtration with $R_x$, we get $(N_i/N_{i-1})_x
\cong (R/Q_i)_x =0$ for all these $i$ (such that $Q_i \not\subseteq P$). 

For the remaining indices $j$ (so that $Q_j \subseteq P$), since
$(M/L)_P$ has finite length over $R_P$, we conclude that
$\length_{R_P}((R/Q_j)_P) < \infty$, which forces $Q_j = P$ in this
case.  Moreover, we see that there are precisely $n$ many indices $j$
such that $Q_j = P$.

Thus, by tensoring 
$L= N_0 \subseteq N_1\subseteq \cdots \subseteq N_r =M$ with $R_x$,
removing the repeated terms and relabeling everything properly, we
obtain a required filtration over $R_x$.  
\end{proof}

\begin{Lma}
\label{koszul}
Let $A \to B$ be a flat homomorphism of Noetherian rings and consider
$$M_0 \subseteq M_1 \subseteq \cdots \subseteq M_n,$$
which is a filtration of $A$-modules. Write $M_i/M_{i-1} = N_i$ for $i = 1,
\dotsc, n$.
Assume that $\mx$ is a sequence in $B$ that is regular on $N_i
\otimes_A B$ for all $i =2, \ldots, n$.  

\begin{enumerate}
\item there is the following natural isomorphism
$$\frac{M_1\otimes B + (\mx) (M_n \otimes B)}{M_0\otimes B +(\mx) 
(M_n \otimes B)}   \cong 
\frac{M_1}{M_0} \otimes \frac{B}{ (\mx) B}.$$

\item assume furthermore that 
$\length _B\left(\frac{N_i \otimes_A B}{(\mx)(N_i \otimes_A B)}\right) 
< \infty$ for $i =1,\dotsc, n$.  Then
$$
\length _B \left(\frac{M_n}{M_0} \otimes_A \frac{B}{(\mx)B}\right) 
= \sum_{i=1}^n \length _B \left(N_i \otimes_A \frac{B}{(\mx)B}\right).
$$ 
\end{enumerate}
\end{Lma}

\begin{proof}
Since $A \to B$ is a flat ring homomorphism, the hypotheses on the
given filtration of $A$-modules are stable under the scalar extension
along $A \to B$. Thus, without loss of generality, we assume $A=B$.
And we choose to use $A$ when presenting the proof.

(1) Note that $\mx$ is a regular sequence on $M_n/M_1$ 
by a routine short exact sequence argument. 
Next, consider the short exact sequence
$$0 \longrightarrow M_1/M_0 \longrightarrow M_n/M_0 \longrightarrow  
M_n/M_1 \longrightarrow 0.$$ 
Tensoring with $A/(\mx)A$, we get
$$0\longrightarrow \frac{M_1}{M_0} \otimes \frac{A}{(\mx)A} 
\longrightarrow \frac{M_n}{M_0} \otimes \frac{A}{(\mx)A} 
\longrightarrow \frac{M_n}{M_1}\otimes \frac{A}{(\mx)A} 
\longrightarrow 0,$$
since the Koszul homology $H_1\big(\mx, \frac{M_n}{M_1}\big)=0$.  
But also note the natural isomorphisms
\[
\frac{M_n}{M_0} \otimes \frac{A}{(\mx)A} \cong
\frac{M_n}{M_0+(\mx)M_n} \quad \text{and} \quad  
\frac{M_n}{M_1} \otimes \frac{A}{(\mx)A} \cong \frac{M_n}{M_1+ (\mx)M_n}.
\]
So using these natural isomorphisms, one sees $\frac{M_1}{M_0} \otimes
\frac{A}{(\mx)A} \cong \frac {M_1+ (\mx)M_n}{M_0+(\mx)M_n}$, which
finishes the proof of this part. 
 
(2) Consider a piece of the filtration, which gives a short exact
    sequence 
$$0 \longrightarrow \frac{M_{i-1}}{M_0} \longrightarrow 
\frac{M_i}{M_0} \longrightarrow \frac{M_{i}}{M_{i-1}} \longrightarrow 0, 
\quad i = 2, \dotsc, n.$$ 
Note that since $\mx$ are a regular sequence on $N_i$, we
get that the Koszul homology $H_1(\mx, N_i)=0$. Applying the
long exact sequence of Koszul homology with respect to $\mx$, we get
an exact sequence as follows 
$$0 \longrightarrow \frac{M_{i-1}}{M_0} \otimes_A \frac{A}{(\mx)A} 
\longrightarrow  \frac{M_i}{M_0} \otimes_A \frac{A}{(\mx)A} 
\longrightarrow  \frac{M_{i}}{M_{i-1}} \otimes_A \frac{A}{(\mx)A}
\longrightarrow  0.$$ 

Using the additivity of the length function and summing over all such
short exact sequences for $2=1,\ldots, n$, we obtain that 
\[
\length_A \left(\frac{M_{n}}{M_0} \otimes \frac{A}{(\mx)A}\right) 
= \sum_{i=1}^n \length _A \left(\frac{M_{i}}{M_{i-1}} \otimes
  \frac{A}{(\mx)A}\right) = \sum_{i=1}^n \length_A
\left(N_i \otimes \frac{A}{(\mx)A}\right). 
\qedhere
\]
\end{proof}

\begin{Thm}
\label{hom}
Let $R$ be a homomorphic image of a regular ring of prime
characteristic $p$ and assume that $R$ is excellent and locally
equidimensional.  

Then the normalized Frobenius splitting numbers are lower
semicontinuous, i.e., $s_e$ is lower semicontinuous for every $e \ge 0$. 
\end{Thm}

\begin{proof}
Using the notations introduced at the beginning of the section we let
$R=S/I$. Let $\overline{P} =P/I$ be a prime ideal of $R$, in which $P$
is a prime ideal of $S$ containing $I$. Quite generally, denote
$\overline Q : = Q/I$ for all $Q \in \Spec(S) \cap V(I)$. 
Let $K = (I \brq :_S I)$, an ideal in $S$.

Let $m := \length _{S_P}
\big(\big(\frac{K+P\brq}{P\brq}\big)_P\big)$. Then
Proposition~\ref{rewrite} allows us to write  
$$\ta_e(\overline{P}) = \length _{S_P}
\left(\left(\frac{K+P\brq}{P\brq}\right)_P\right) 
\cdot \frac{1}{q^{\height(P/I)}} = m \cdot \frac{1}{q^{\height(P/I)}}.$$ 

Consider ${P\brq} \subseteq {K+P\brq} \subseteq {S}$. Applying
Lemma~\ref{filtration} to ${P\brq} \subseteq {K+P\brq}$ and ${K+P\brq}
\subseteq {S}$, we can find $x$ in $S \setminus P$ and a filtration  
$${P\brq}_x=M_0 \subseteq M_1 \subseteq \cdots \subseteq M_m \subseteq
M_{m+1} \subseteq \cdots \subseteq M_{m+n} =S_x$$ 
such that $M_m = ({K+P\brq})_x$ and $M_{i}/M_{i-1} \cong
(S/P)_x$ for all $i =1, \ldots, m+n$. 

Since $R$ is excellent, the regular locus of $S/P$ is open (and
non-empty), so there 
exists $y \notin P$ such that $R/\overline{P} = S/P$ becomes regular
when localizing at $y$. We can replace $x$ by $xy$ and hence simply
assume that $(S/P)_x$ is in fact regular. 

Consider any $Q \in D(x) \subseteq \Spec(S)$ such that $P \subseteq
Q$. Since $(S/P)_x$ is regular we obtain that $(S/P)_Q$ is regular as
well. Let us choose  $\mx = x_1, \ldots, x_d$ in $S$ 
such that their images in $(S/P)_Q$ form a regular system of
parameters (of the regular local ring $(S/P)_Q$). More explicitly, we
have $Q_Q = (P+(\mx))_Q$ and $d =
\dim((R/P)_Q) = \height(Q/P)$. 

Note that Proposition~\ref{rewrite} also allows us to write
$$\ta_e (Q/I) = \length_{S_Q}
\left(\left(\frac{K+Q\brq}{Q\brq}\right)_Q\right) \cdot
        \frac{1}{q^{\height(Q/I)}} 
= \length_{S_Q} \left(\left(
\frac{K+(P + (\mx)) \brq}{P\brq+ (\mx)\brq}\right)_Q\right) 
\cdot \frac{1}{q^{\height(Q/I)}}.$$ 

Remember that $m = \length _{S_P}
\big(\big(\frac{K+P\brq}{P\brq}\big)_P\big)$ and
$d = \height (Q/P)$, which equals $\Dim(R_{\overline Q}) -
\Dim(R_{\overline P})$
since $R$ is locally equidimensional. 

Our plan is to show that
\[ \length_{S_Q} \left(\left(
\frac{K+(P\brq +(\mx) \brq)}{P\brq+ (\mx) \brq}\right)_Q\right)  
= m \cdot 
\length_{S_Q} \left(\left(\frac{S}{P+(\mx)\brq}\right)_Q\right) 
= m \cdot q^d.
\tag{$*$}
\] 
This will show that $\ta_e(Q/I) = \ta_e(P/I)$.

In order to prove our claim let us return to the filtration we
considered above 
$${P\brq}_x=M_0 \subseteq M_1 \subseteq \cdots \subseteq M_m \subseteq 
M_{m+1} \subseteq \cdots \subseteq M_{m+n} =S_x,$$ 
in which $M_m = ({K+P\brq})_x$ and $M_{i}/M_{i-1} \cong
(S/P)_x$ for all $i =1, \ldots, m+n$. 

Applying Lemma~\ref{koszul} (1) to the filtration $M_0 \subseteq M_m
\subseteq \cdots \subseteq M_{m+n}$ with $A=S_x \to B=S_Q$ and the
sequence $\mx \brq:= x_1^q, \dotsc, x_d^q$ which is regular on
$(S/P)_Q$, we get 
\[
\left(\frac{(K+P\brq)  +(\mx) \brq}{P\brq+ (\mx) \brq}\right)_Q \cong 
\frac{({K+P\brq})_x}{{P\brq}_x} \otimes \frac{S_Q}{ (\mx)\brq S_Q}.
\tag{$\dagger$}
\]

Then, applying Lemma~\ref{koszul} (2) to the filtration $M_0 \subseteq
M_1 \subseteq \dotsb \subseteq M_{m}$, we see
\[
\length_{S_Q} \left(\tfrac{({K+P\brq})_x}{{P\brq}_x} \otimes
  \tfrac{S_Q}{ (\mx)\brq S_Q}\right) 
=\sum_{i=1}^m \length_{S_Q} \left( \left(\tfrac S P \right)_x \otimes
  \tfrac{S_Q}{ (\mx)\brq S_Q} \right) 
= m \cdot \length_{S_Q} \left(\tfrac{S_Q}{(P+ (\mx)\brq) S_Q} \right).
\]

Since $\mx$ are a regular system of parameters on $(S/P)_Q$, we deduce
$\length_{S_Q} \left(\frac{S_Q}{(P+(\mx)\brq) S_Q} \right) =
q^d = q^{\height(Q/P)}$  and hence
\[
\length_{S_Q} \left(\frac{({K+P\brq})_x}{{P\brq}_x} \otimes
  \frac{S_Q}{ (\mx)\brq S_Q}\right) = m \cdot q^d 
= m \cdot q^{\height(Q/P)}. 
\tag{$\ddagger$}
\]

Combining $(\dagger)$ and $(\ddagger)$, we see our claim $(*)$ that 
$\length_{S_Q} \Big(\big(
\frac{K+(P\brq +(\mx) \brq)}{P\brq+ (\mx) \brq}\big)_Q\Big)  
= m \cdot q^{\height(Q/P)}$, which then implies
\begin{align*}
s_e(Q/I) &= \length_{S_Q} \left(\left(
\frac{K+(P + (\mx)) \brq}{P\brq+ (\mx)\brq}\right)_Q\right) 
\cdot \frac{1}{q^{\height(Q/I)}} \\
&= m \cdot q^{\height(Q/P)} \cdot \frac{1}{q^{\height(Q/I)}}
= \length_{S_P}\left(\left(\frac{K+P\brq}{P\brq}\right)_P\right) 
\cdot \frac{1}{q^{\height(P/I)}} = s_e(P/I).
\end{align*}

In summary, 
for any given prime ideal $P/I$ of
$R$ (i.e., $P \in V(I) \subseteq \Spec(S)$), there is $x \in S
\setminus P$ such that $\ta_q(Q/I) = \ta_e(P/I)$ for all $Q \in D(x)
\cap V(P)$. 

Now, to prove the lower semicontinuity of $s_e : \Spec(R) \to \mathbb
Q$, let $r \in \mathbb R$ 
be a number and let $U = \{ \overline{Q} : \ta_e(\overline{Q}) > r \}
\subseteq \Spec(R).$  To show that $U$ is open it is enough to apply
Proposition~\ref{nagata}. 

Part (i) in Proposition~\ref{nagata} is satisfied by \cite
[Proposition~5.2] {Y} where it is shown that the normalized Frobenius
splitting numbers can only increase by localization over a 
locally equidimensional ring. 

Part (ii) follows from what we just proved earlier: For any $\overline
P = P/I \in U$, the work right above shows $
D(\overline{x}) \cap V(\overline P) \subseteq U$, in which $x \in S
\setminus P$ is as obtained above and $\overline{x}$ denotes the image
of $x$ in $R$. This completes the proof.
\end{proof}

\begin{Rem} \label{rmk-hom}
The reader should note that under the conditions of Theorem~\ref{hom},
the same proof shows that the set $X_{\geq r} = \{ Q \in \Spec(R) :
\ta_e( Q) \geq r \}$ is open for all $r \in \mathbb R$.
\end{Rem}

\section{Gorenstein excellent rings}

We state the following lemma about Gorenstein rings. 

\begin{Lma}
\label{description}
Let $(R, \fm, k)$ be a local Gorenstein ring of prime positive
characteristic $p$ and dimension $d$. Let $\mx =x_1, \ldots, x_d$ be
any system of parameters for $R$. 

Then, for any $u \in R$ such that
its image generates the socle of $R/(\mx)$, we have
$$
\ta_e (R) \cdot q^d = \length \left(\frac{Ru^q+ (\mx) \brq}{(\mx)
    \brq}\right) \quad \text{for all} \quad e \ge 0.
$$ 
(When $\dim(R) = 0$, we agree that $(\mx) = 0$ by convention.)
\end{Lma}

\begin{proof}
Note that the statement is clear when $\dim(R) = 0$. 
(In fact, the following proof also works for the case of
$\dim(R) = 0$ if we agree that $x = 1$ and $(\mx)^{[n]} = 0$
for all $n \ge 0$.) 

Denote $x:= \prod_{i=1}^d x_i$ and 
$(\mx)^{[n]}: = (x_1^n, \dotsc, x_d^n)$ for all $n \ge 0$. 
As $R$ is Gorenstein, the injective hull of $k = R/\fm$ can be
obtained as  
\[
E=E_R(k) = \varinjlim
\left\{\frac{R}{(\mx)} \overset{\cdot x}{\longrightarrow} 
\frac{R}{(\mx)^{[2]}} \overset{\cdot x}{\longrightarrow} 
\dotsb \overset{\cdot x}{\longrightarrow} 
\frac{R}{(\mx)^{[n]}} \overset{\cdot x}{\longrightarrow} 
\frac{R}{(\mx)^{[n+1]}} \overset{\cdot x}{\longrightarrow} \dotsb
\right\}.
\]
Note that all the maps in the direct system above are injective.

Consider the $R$-linear injective map $f: k = \frac{R}{\fm} 
\overset{\cdot u}{\rightarrow} \frac{R}{(\mx)}$ defined by 
$f(r+\fm) = ru +(\mx)$ for $r \in R$,  which induces
an injective $R$-linear map $\psi: k \to E$ as all the maps in the
direct system above are injective.  

By Definition~\ref{def}, we see   
$$
\ta_e (R) \cdot q^d = \length_{\R{e}} \left(\im\left(\R{e} \otimes _R
k \stackrel{1\otimes\psi}\longrightarrow \R{e}
\otimes_R E\right)\right). 
$$

Moreover, by the property of tensor product and direct limit, we see 
$$
\R{e} \otimes_R E =
\varinjlim
\left\{\R{e} \otimes \tfrac{R}{(\mx)} 
\overset{1 \otimes (\cdot x)}{\longrightarrow} 
\dotsb \overset{1 \otimes (\cdot x)}{\longrightarrow} 
\R{e} \otimes \tfrac{R}{(\mx)^{[n]}} 
\overset{1 \otimes (\cdot x)}{\longrightarrow} 
\R{e} \otimes \tfrac{R}{(\mx)^{[n+1]}} 
\overset{1 \otimes (\cdot x)}{\longrightarrow} \dotsb 
\right\}.
$$
Thus, the image of the homomorphism $\R{e} \otimes _R
k \stackrel{1\otimes\psi}\longrightarrow \R{e}
\otimes_R E$ is exactly the image of $\R{e} \otimes _R
\tfrac R{\fm}$ in the direct limit of the following direct
system 
\[
\R{e} \otimes \tfrac{R}{\fm} 
\overset{1 \otimes (\cdot u)}{\longrightarrow} 
\R{e} \otimes \tfrac{R}{(\mx)} 
\overset{1 \otimes (\cdot x)}{\longrightarrow} 
\dotsb \overset{1 \otimes (\cdot x)}{\longrightarrow} 
\R{e} \otimes \tfrac{R}{(\mx)^{[n]}} 
\overset{1 \otimes (\cdot x)}{\longrightarrow} 
\R{e} \otimes \tfrac{R}{(\mx)^{[n+1]}} 
\overset{1 \otimes (\cdot x)}{\longrightarrow} \dotsb .
\]
However, by the meaning of $\R{e}$, the above direct system can be
written as 
\[
\frac{R}{\fm\brq} \overset{\cdot u^q}{\longrightarrow} \frac{R}{(\mx)\brq} 
\overset{\cdot x^q}{\longrightarrow} \dotsb 
\overset{\cdot x^q}{\longrightarrow} \frac{R}{(\mx)^{[nq]}} 
\overset{\cdot x^q}{\longrightarrow} \frac{R}{(\mx)^{[n+1]q}} 
\overset{\cdot x^q}{\longrightarrow} \dotsb,
\]
in which 
$\tfrac{R}{(\mx)^{[nq]}} \stackrel{\cdot x^{q}}\longrightarrow
\tfrac{R}{(\mx)^{[(n+1)q]}}$ is injective for every $n$ since $\mx$ is
regular on $R$. 

Putting things together, we see 
\begin{align*}
\ta_e (R) \cdot q^d 
&= \length_{\R{e}} \left(\im\left(\R{e} \otimes _R k 
\stackrel{1\otimes\psi}\longrightarrow \R{e} \otimes_R E\right)\right)\\
&= \length_R \left(\im\left(
\frac{R}{\fm\brq} \overset{\cdot u^q}{\longrightarrow} \frac{R}{(\mx)\brq} 
\right)\right)
=\length_R \left(\frac{Ru^q+ (\mx) \brq}{(\mx) \brq}\right).
\qedhere
\end{align*} 
\end{proof}

\begin{Thm}
\label{gor}
Let $R$ be an excellent Gorenstein ring. Then the normalized Frobenius
splitting numbers are lower semicontinuous, i.e.,  $s_e$ is lower
semicontinuous for every $e \ge 0$. 
\end{Thm}

\begin{proof}
Let $r \in \mathbb R$ and let $P$ be a prime ideal in $R$ such that
$\ta_e(P) > r$. 

By prime avoidance, we can choose $\mx = x_1, \dotsc, x_k$ that is a
regular sequence on $R$ such that their images form a system of 
parameters on $R_P$, where $k = \height(P)$. Let $u \in R$ be such
that its image in $R_P$ generates the socle modulo $(\mx)$, the ideal
generated by $\mx$.  

Let $m := \length_{R_P}\left(\big(\frac{(u^q) + (\mx) \brq}{(\mx
  \brq)}\big)_P\right)$. By Lemma~\ref{description} we get that  
$$
m = \length_{R_P}\left(\Big(\frac{(u^q) + (\mx) \brq}{(\mx
  \brq)}\Big)_P\right)= \ta_e(P) \cdot q^{\height(P)}.
$$

Now we proceed as in the proof of Theorem~\ref{hom}. 
Applying Lemma~\ref{filtration} to $(\mx)\brq
\subseteq (u^q) + (\mx)\brq \subseteq R$ and $(\mx)
\subseteq (u) + (\mx) \subseteq R$, we therefore obtain $x \in R
\setminus P$ and filtrations
\begin{gather}
{(\mx)\brq}_x =M_0 \subseteq M_1 \subseteq \dotsb \subseteq M_m 
\subseteq M_{m+1} \subseteq \dotsb \subseteq M_{m+n}= R_x 
\tag{\ref{gor}.1} \label{filtration1}\\
\text{and} \quad 
{(\mx)}_x =N_0 \subseteq N_1 \subseteq \dotsb \subseteq  N_t = R_x,
\tag{\ref{gor}.2} \label{filtration2}
\end{gather}
in which $M_m = ((u^q) + (\mx) \brq)_x$, $N_1 = (Ru+ (\mx))_x$, and
$M_i/M_{i-1} \cong N_j/N_{j-1} \cong (R/P)_x$ for all $i = 1,
\dotsc, m+n$ and all $j = 1, \dotsc, t$.

Similarly, since $R$ is excellent and thus the regular locus of $R/P$
is open, we may just as well further assume that $(R/P)_x$ is regular.

Let $Q$ be any prime ideal in $D(x) \cap V(P)$. As $(R/P)_Q$ is
regular, let $\my = y_1,\ldots, y_h$ be chosen such that their images
in $(R/P)_Q$ form a regular system of parameters. (In particular, we
have $Q_Q = (P+(\my))_Q$ and $h = \height (Q/P)$.)
It then follows that $\my$ is a regular
sequence on $(R/(\mx))_Q$ (because of the
filtration~\eqref{filtration1}). Thus $\mx, \my$ form a system of
parameters of $R_Q$. 

As in the proof of Theorem~\ref{hom}, we apply Lemma~\ref{koszul}
to the filtration~\eqref{filtration1} and $R_x \to R_Q$ to deduce that
$$
\length _{R_Q} \left( \left(\frac{(u^q) +(\mx) \brq +(\my) \brq}
{(\mx) \brq +(\my) \brq}\right)_Q\right) = m \cdot q^h
= m \cdot q^{\height (Q/P)}.
$$

Next, we show that the image of $u$ generates the socle of
$\big(\frac R{(\mx, \my )}\big)_Q$: 
Note that $R_Q$ is Gorenstein and $QR_Q = (P+(\my))R_Q$. 
Applying Lemma~\ref{koszul}~(1)
to the filtration~\eqref{filtration2},  $R_x \to R_Q$ and the sequence
$\my$ that is regular on $(R/P)_Q$, we deduce that the image of $u$
generates the following submodule in
$\big(\frac R{(\mx, \my )}\big)_Q$
\begin{multline*}
\left(\frac{(u)+(\mx) +(\my)}{(\mx) +(\my)}\right)_Q =
\left(\frac{N_1 + (\my)N_t}{N_0 +(\my)N_t}\right)_Q \\
\cong \frac{N_1}{N_0} \otimes \frac{R_Q}{(\my)R_Q} \cong 
\frac{R}{P} \otimes \frac{R_Q}{(\my)R_Q} =
\frac{R_Q}{(P + (\my))R_Q} \cong \frac{R_Q}{QR_Q}.
\end{multline*}
Thus, as $R_Q$ is Gorenstein, the image of $u$ generates the socle of
$\big(\frac R{(\mx, \my )}\big)_Q$.

Applying again Lemma~\ref{description} to $R_Q$ and $\mx, \my$, we get
that  
$$
\ta_e(Q) \cdot q ^{\height(Q)} 
=m \cdot q^{\height (Q/P)} 
= \ta_e(P) \cdot q^{\height(P) + \height(Q/P)} .$$ 
Note that $\height(P) + \height(Q/P)  =\height(Q)$ because $R$ is
Gorenstein, hence locally  equidimensional. 
So in summary, we see 
$$\ta_e(Q) =\ta_e(P) \quad \text{for all } Q \in D(x) \cap V(P).$$

Finally, we show the lower semicontinuity of the normalized Frobenius
splitting numbers by applying Proposition~\ref{nagata}. Note that the
normalized Frobenius splitting number can only increase under
localization by in \cite [Proposition~5.2] {Y}. Thus we need to check
only part~(ii) of Proposition~\ref{nagata} for $U = \{ Q : \ta_e(Q) >
r \}$. But part~(ii) is clear by now since, for any $P \in U$, there
is $x \in R \setminus P$ such that $D(x) \cap V(P) \subseteq U$ by the
work above. It follows that $U = \{ Q : \ta_e(Q) > r \}$ is open.   
\end{proof}

\begin{Rem}
The reader should note that under the conditions of Theorem~\ref{gor},
the same proof shows that the set $X_{\geq r} = \{ Q \in \Spec(R) :
\ta_e( Q) \geq r \}$ is open for all $r \in \mathbb R$.
\end{Rem}

\section{Essentially of finite type over a semi-local excellent ring}

In the following theorem, $\widehat A$ denotes the completion of a
semi-local ring $A$ with respect to its Jacobson radical. Note that
$\widehat A$ is 
isomorphic to $\prod_{\fm} \widehat{A_{\fm}}$ in which $\fm$ runs over
all the (finitely many) maximal ideals $\fm$ of $A$ while
$\widehat{A_{\fm}}$ is the $\fm$-adic completion of $A$ (or of
$A_{\fm}$). In particular, $\widehat A$ is a homomorphic image of a
regular ring.

\begin{Thm}
\label{ess}
Let $R$ be a ring of prime characteristic $p$ that is essentially of
finite type over an excellent semi-local ring $A$. Assume that $R
\otimes_A \widehat{A}$ is locally equidimensional. 

Then the normalized Frobenius splitting numbers are lower
semicontinuous, i.e., $s_e$ is lower semicontinuous for every $e \ge 0$.
\end{Thm}

\begin{proof}
Let us fix a nonnegative integer $e$ and let $r \in \mathbb R$ be any
number. We plan to show that the set $ W= \{ Q : Q \in \Spec(P),\,
\ta_e(Q) > r \}$ is open by passing to $R \otimes_A \widehat{A}$,
where Theorem~\ref{hom} applies. Fix a prime ideal $P \in W$, i.e.,
$s_e(P) > r$.  

Let $S = R \otimes_A \widehat{A}$ and let $C = \{P' \in \Spec(S) :
s_e(P') \le r\}$. Note that Theorem~\ref{hom} applies to $S$ since $S$
is essentially of finite type over $\widehat A$. 
Therefore, the normalized Frobenius splitting numbers are lower
semicontinuous on $S$. Thus, $C$ is closed in $\Spec(S)$ so there
exists an ideal $I$ of $S$ such that $V(I) = C$. 

We claim that $I \cap (R\setminus P) \neq \emptyset$: By way of
contradiction, suppose $I \cap (R\setminus P) = \emptyset$. Then there
exists a prime ideal $P' \in \Spec(S)$ such that $P' \in V(I) = C$ and
$P' \cap (R\setminus P) = \emptyset$, which simply means $P' \cap R
\subseteq  P$. Denote $P_1 := P' \cap R \in \Spec(R)$.
Now, since $A$ is excellent, we see that $R \to S$ has geometrically
regular fibers (cf. \cite[33.E~Lemma~4]{M2}) 
and hence the map $R_{P_1} \to S_{P'}$ has a regular closed 
fiber. Thus, by \cite[Theorem~5.6]{Y}, we see
$\ta_e(P') = \ta_e(P_1)$. Moreover, by \cite[Proposition~5.2]{Y} and
also noting that $R$ is locally equidimensional (since $S$ is so), we
have $\ta_e(P_1) \geq \ta_e(P)$. Consequently, $\ta_e(P') \geq
\ta_e(P) > r$, contradicting the choice of $P'$ with $P' \in C$.

Now that $I \cap (R\setminus P) \neq \emptyset$, let $x \in I \cap (R
\setminus P)$ and consider $Q \in \Spec(R)$ such that $x \notin Q$
(that is, $Q \in D(x) \subseteq \Spec(R)$).  

As $S$ is faithfully flat over $R$, there exists $Q' \in \Spec(S)$
such that $Q' \cap R =Q$. Hence $x \notin Q'$ and thus $I \not\subseteq
Q'$. This says that $Q' \notin V(I) = C$, or simply, $s_e(Q') > r$. 

We can see that $R_Q \to S_{Q'}$ is faithfully flat and so $\ta_e(Q)
\geq \ta_e(Q')$ by \cite[Lemma~5.1]{Y}. (Or use \cite[Theorem 5.6]{Y}
to deduce that $\ta_e(Q) = \ta_e(Q')$.) Thus $\ta_e(Q) > r$.

Putting everything together, we get $\ta_e(Q) > r$ for all $Q \in
D(x)$. That is 
$$
P \in D(x)  \subseteq W.
$$

By the choice of $x$, $P$ lies in the above open neighborhood
$D(x)$. This shows that the normalized Frobenius
splitting numbers are lower semicontinuous on $R$. 
\end{proof}

By Ratliff's theorem, the completion of an equidimensional excellent
local ring remains equidimensional 
(see \cite[Corollary~B.4.3 and Theorem~B.5.1]{SH}). We have the following

\begin{Cor}
If $R$ is an excellent locally equidimensional semi-local ring (e.g.,
$R$ is an excellent equidimensional local ring) of prime 
characteristic $p$, the function $s_e: \Spec(R) \to \mathbb Q$ is
lower semicontinuous. 
\end{Cor}

\begin{Rem}
The reader should note that under the conditions of Theorem~\ref{ess},
the proof shows that the set $X_{\geq r} = \{Q \in \Spec(R) :
\ta_e( Q) \geq r \}$ is open for all $r \in \mathbb R$. 
(To do this, one accordingly defines $C = \{P' \in \Spec(S) :
s_e(P') <  r\} =  V(I)$ because of Remark~\ref{rmk-hom}.)
\end{Rem}

\begin{Rem} \label{ess-gamma}
Note that the proof of Theorem~\ref{ess} relies on the lower
semicontinuity of $s_e$ on the ring $S = R \otimes_A \widehat A$.
Although this is covered in Theorem~\ref{hom}, we would like to
point out that the lower semicontinuity of $s_e$ on $S$ can also be
proved via $\Gamma$-construction (see \cite{HH}). By applying
$\Gamma$-construction, one may reduce the lower semicontinuity of
$s_e$ on $S$ to the F-finite case and therefore
Corollary~\ref{F-finite} applies. 
\end{Rem}

\begin{Rem}
In general, one may study the notion of \emph{$e^{th}$ normalized Frobenius
  splitting numbers} of any finitely generated module $M$. 
Indeed, for any local ring $(R,\fm,k)$ of prime 
characteristic $p$ and any finitely generated module $R$-module $M$,
one may define $s_e(M):= \frac{\#(^e\! M)}{q^{\dim(R)}}$, in which
$\#(^e\! M)$ is the same as defined in
\cite[Definition~2.2]{Y}. 

Thus, for any fixed $e \ge 0$, there is a
function $s_{e,M}: \Spec(R) \to \mathbb Q$ defined by $P \mapsto
s_e(M_P)$. Note that $s_{e,R} = s_e$ when $M=R$. 
Correspondingly, one may conjecture that $s_{e,M}$ is lower
semicontinuous if the ring is excellent and locally equidimensional
(cf. Conjecture~\ref{conj}).   
In fact, the lower semicontinuity of $s_{e,M}$ can be verified in 
the cases when $R$ is as in Corollary~\ref{F-finite}, as in
Theorem~\ref{hom}, as in Theorem~\ref{gor} and with $M$ being 
(locally) maximal Cohen-Macaulay, or as in Theorem~\ref{ess}. 
Note that when $R$ is as in Theorem~\ref{ess}, one may \emph{also}
prove the lower semicontinuity of $s_{e,M}$ via the approach described
in Remark~\ref{ess-gamma}.  
\end{Rem}


\end{document}